\documentclass[a4paper,reqno]{amsart}

\usepackage{amsmath,amssymb,amsthm}
\usepackage[foot]{amsaddr}
\usepackage{accents}
\usepackage{graphicx}
\usepackage{cite}
\usepackage{mathtools}
\usepackage{nicefrac}
\usepackage{bm}
\usepackage{booktabs}
\usepackage{hyperref}
\usepackage{color}

\def\A{\mathsf{A}}

\def\M{\mathcal M}

\def\cG{\mathsf{FEM}}

\newcommand{\dif}{\,\text{d}}

\renewcommand{\phi}{\varphi}

\renewcommand{\P}{\Upsilon^{r_m}_{m}}

\newcommand{\timejump}[1]{\lbrack\!\lbrack#1\rbrack\!\rbrack} 
\DeclareMathOperator{\diam}{diam}

\makeatletter
\DeclareRobustCommand\widecheck[1]{{\mathpalette\@widecheck{#1}}}
\def\@widecheck#1#2{%
    \setbox\z@\hbox{\m@th$#1#2$}%
    \setbox\tw@\hbox{\m@th$#1%
       \widehat{%
          \vrule\@width\z@\@height\ht\z@
          \vrule\@height\z@\@width\wd\z@}$}%
    \dp\tw@-\ht\z@
    \@tempdima\ht\z@ \advance\@tempdima2\ht\tw@ \divide\@tempdima\thr@@
    \setbox\tw@\hbox{%
       \raise\@tempdima\hbox{\scalebox{1}[-1]{\lower\@tempdima\box
\tw@}}}%
    {\ooalign{\box\tw@ \cr \box\z@}}}
\makeatother

\newtheorem{Remark}[equation]{Remark}
\newenvironment{remark}{\begin{Remark}\rm}{\end{Remark}}
\newtheorem{theorem}[equation]{Theorem}
\newtheorem{assumption}[equation]{Assumption}
\newtheorem{proposition}[equation]{Proposition}
\newtheorem{corollary}[equation]{Corollary}
\newtheorem{lemma}[equation]{Lemma}

\numberwithin{equation}{section}

\newcommand{\C}{\mathrm{C}}
\newcommand{\W}{\mathrm{W}}
\renewcommand{\H}{\mathrm{H}}
\renewcommand{\P}{\mathbb{P}}
\newcommand{\diff}{\kappa}
\newcommand{\Ltwo}{\mathrm{L}^2}
\newcommand{\Linf}{\mathrm{L}^\infty}
\newcommand{\bigcdot}{\bullet}

\title[Conditional high order error bounds for heat blow-up models]{%
Conditional a posteriori error bounds\\ 
for high order DG time stepping approximations\\
of semilinear heat models with blow-up
}
\author{Stephen Metcalfe}
\address{Dept. of Mechanical Engineering, McGill University, Montr{\'e}al, H3A 0C3, Canada}
\email{smetcalfephd@gmail.com}
\author{Thomas P. Wihler}
\address{Mathematisches Institut, Universit\"at Bern, Sidlerstr.~5, CH-3012 Bern, Switzerland}
\email{thomas.wihler@math.unibe.ch}

\thanks{The authors acknowledge the support of the Swiss National Science Foundation (SNF) grant \#200021-162990. In addition, the results in this paper made use of the facilities of Compute Canada and Calcul Qu{\'e}bec, specifically, the B{\'e}luga supercomputer.}

\begin{document}

\begin{abstract}
This work is concerned with the development of an adaptive numerical method for semilinear heat flow models featuring a general (possibly) nonlinear reaction term that may cause the solution to blow up in finite time. The fully discrete scheme consists of a high order discontinuous Galerkin (dG) time stepping method and a conforming finite element discretisation (cG) in space. The proposed adaptive procedure is based on rigorously devised \emph{conditional a posteriori error bounds} in the $\Linf(\Linf)$ norm. Numerical experiments complement the theoretical results.
\end{abstract}



\maketitle

\section{Introduction}

Let $\displaystyle \Omega \subset \mathbb{R}^d$ with $d=1$, $2$ or $3$ be a bounded polyhedral domain and consider the initial boundary value problem
\begin{equation}\label{model_strong}
\begin{aligned}
u_t - \diff\Delta{u}&=f(u)  \qquad && \text{in }  \Omega, \mbox{ } t > 0  \mbox{,} \\ u &=0 \mbox{ } && \text{on }  \partial\Omega, \mbox{ } t > 0 \mbox{,} \\ u(\cdot,0) &=u_0 \mbox{ } && \text{in } \overline{\Omega}\mbox{,}
\end{aligned}
\end{equation}
where $\diff>0$ is a constant diffusion coefficient and $u_0 \in \W^{2, \infty}(\Omega)$ is the initial condition with $u_0 |_{\partial \Omega} = 0$. We assume that the reaction term $f:\overline{\Omega} \times \mathbb{R}_0^+\times \mathbb{R}\to \mathbb{R}$ is continuously differentiable and satisfies the \emph{local} Lipschitz estimate
\begin{equation}\label{eq:Lip}
|f(x,t,v)-f(x,t,w)| \leq  \mathfrak{L}(t,|v|,|w|)|v-w| \quad \forall x \in \overline{\Omega} \quad \forall t \in \mathbb{R}_0^+ \quad \forall v,w \in \mathbb{R}.
\end{equation}
Here, $\mathfrak{L}:\mathbb{R}_0^+\times \mathbb{R}^+_0 \times \mathbb{R}^{+}_0 \to  \mathbb{R}^{+}_0$ is a \emph{known} function that satisfies $\mathfrak{L}(\cdot,a,b) \in \mathrm{L}^1_{\text{loc}}(\mathbb{R}_0^+)$ for any $a,b \in \mathbb{R}^+_0$, and that is continuous and monotone increasing in the second and third arguments. This condition on $f$ is quite general and includes many nonlinearities of interest, for example, it covers any polynomial nonlinearity with suitably regular coefficients as well as nonlinearities of exponential type \cite{KMW16}.

If the reaction term $f$ features sufficient growth, and if the initial data $u_0$ possesses enough energy, then it is known that \eqref{model_strong} exhibits \emph{finite time blow-up}, that is, there exists a maximal time of existence $T_\infty < \infty$ called the \emph{blow-up time} such that \eqref{model_strong} holds and
\[
\|u(t)\|_{\Linf(\Omega)}<\infty\text{ for }0<t<T_\infty,\qquad \qquad \lim_{t\nearrow T_\infty}\|u(t)\|_{\Linf(\Omega)} = \infty,
\]
see, e.g., \cite{Hu} and the references presented therein. If the solution to \eqref{model_strong} does not exhibit finite time blow-up then the solution is global and $T_{\infty} = \infty$.  Either way, we assume that \eqref{model_strong} has a solution on the time interval $(0,T_\infty)$. In fact, we can show that the model problem \eqref{model_strong} has a unique local solution $u \in \Linf(\Linf)$ provided that an implicit \emph{local a posteriori criterion} is satisfied and that this local criterion is well behaved with respect to the distance from the blow-up time.

\emph{A posteriori} error estimators for linear problems are \emph{unconditional} in the sense that the error bounds always hold independently of the discretisation parameters and the problem data. By contrast, \emph{a posteriori} error estimators for nonlinear problems are often \emph{conditional}, that is, the error bounds only hold under the provision that some \emph{a posteriori} verifiable condition is fulfilled. Most of the conditional estimates in the literature are \emph{explicit} \cite{CGS20, B05,CM08,FV03, LN05,MN06,KNS04,K09,KM11,GM14} in the sense that the estimates only hold under conditions of an explicit nature involving the magnitude of the numerical solution, the discretisation parameters and/or the problem data. Recently, there has been interest in the derivation of \emph{implicit} conditional estimates, cf. \cite{KMW16,IM17,M15,CGKM16}, where estimates only hold under conditions that involve the above listed arguments in an implicit manner.For nonlinear time-dependent problems, there are two commonly used approaches for deriving conditional \emph{a posteriori} error bounds: continuation arguments, cf. \cite{B05, M15,KMW16,CGKM16, KNS04,GM14}, and fixed point arguments, cf. \cite{KM11, CM08, K09, IM17}.

A particular asset of deriving conditional \emph{a posteriori} error bounds for blow-up problems is that they can only hold in the pre-blowup phase when the error remains bounded by a finite quantity. If such estimators are robust with respect to the distance from the blow-up time then, when appropriately combined with a sensible adaptive strategy, we can potentially guide the numerical solution process accurately towards the blow-up time. First attempts to derive such error bounds for blow-up problems were made in \cite{K09,KM11} whereby explicit conditional $\Linf(\Linf)$ \emph{a posteriori} error bounds have been derived for \eqref{model_strong} via semi-group theory. These early conditional estimates, however, are not well suited for the practical computation of blow-up problems due to the conditions being of an explicit rather than implicit nature. In view of more practical error bounds for blow-up problems, which involve an implicit condition, a major advancement was presented in \cite{CGKM16,M15} via an energy argument and with the aid of the Gagliardo-Nirenberg inequality for an analysis in the $\Ltwo(\H^1)$ norm; the use of energy techniques to derive the conditional \emph{a posteriori} error bounds in those works also allows for the consideration of non-symmetric spatial operators which is currently out of reach for derivations based on the semi-group methodology. Nevertheless, the bounds of \cite{CGKM16,M15} still possess certain disadvantages: convergence to the blow-up time is slower than was anticipated; moreover, the use of an energy argument means that the range of possible nonlinearities is restricted to those which scale like polynomials of degree up to~3.  Recently, significant progress has been made by combining the implicit condition approach pioneered in \cite{CGKM16,M15} with the original idea of deriving error bounds for blow-up problems via semi-group techniques first presented in\cite{K09,KM11}; this has led to the derivation in \cite{IM17} of an implicit conditional $\Linf(\Linf)$ \emph{a posteriori} error bound for a first order in time implicit-explicit (IMEX) discretisation of \eqref{model_strong}. By combining these approaches in \cite{IM17}, the range of possible nonlinearities that can be considered has been significantly broadened and the rate of convergence to the blow-up time improved from that observed in \cite{CGKM16,M15}. We remark that other $\Linf(\Linf)$ \emph{a posteriori} error bounds for \eqref{model_strong} also exist including those of \cite{KL16, KL17} which focus on the singularly perturbed case and those of \cite{DLM09} which focus on the linear case.

The primary goal of the current paper is to extend the low order IMEX analysis in \cite{IM17} through the incorporation of the high order Galerkin time stepping framework of \cite{KMW16} in order to produce an implicit conditional \emph{a posteriori} error bound for arbitrarily high order  discretisations of \eqref{model_strong}. In particular, the variable order framework allows us to resolve the underlying solution quite efficiently by employing larger time steps and approximation orders in the pre-blowup regime (where the solution is smooth and spectral convergence can be exploited) before transitioning to low order approximations on shorter time intervals close to the blow-up time; this is inspired by the $hp$-version dG time stepping approach for linear parabolic problems with start-up singularities originally presented in \cite{SS00}. Galerkin discretisations of initial value problems are based on weak formulations where the test spaces consist of polynomials that are discontinuous at the time nodes. In this way, the discrete Galerkin formulations decouple into local problems on each time step and the discretisations can therefore be understood as implicit one-step schemes. In the literature, Galerkin time stepping schemes have been extensively analyzed for ODEs, cf. ~\cite{BR03,DD86,DHT81,E95,EF94,J88}. In addition to our proposed method of approximating the solution to \eqref{model_strong} close to the blow-up time through \emph{a posteriori} error estimation, other numerical methods are available; of particular prominence are the rescaling algorithm of Berger and Kohn \cite{BK88,NZ16} and the MMPDE method \cite{BHR96,HMR08}. We also note the classical work of Stuart and Floater \cite{SF90} which deals with the numerical approximation of blow-up in ODEs.

\subsection*{Outline} In \S\ref{sc:disc}, we introduce necessary notation and present the dG-cG discretisation of \eqref{model_strong} which consists of discontinuous Galerkin (dG) time stepping combined with the conforming finite element method (cG) in space. In \S\ref{sc:rec} we present temporal and spatial reconstructions which are crucial for the derivation of our conditional \emph{a posteriori} error bound in \S\ref{sc:erroranalysis}. In \S\ref{sc:adaptive}, we discuss an adaptive algorithm which can exploit the conditional \emph{a posteriori} error bound of the previous section to direct the numerical solution towards the blow-up time. This adaptive algorithm is then applied to several test problems in \S\ref{sc:num} and we are even able to obtain exponential convergence results with an $hp$-adaptive strategy. Finally, we draw conclusions and discuss possible directions for future research in \S\ref{sc:concl}.

\section{Galerkin Discretisation} \label{sc:disc}

\subsection{Spatial Discretisation}

Consider a shape-regular mesh $\mathcal{T}=\{K\}$ of the domain $\Omega$ into open elements $K$ of diameter $h_K:=\diam(K)$ that are constructed via affine mappings $F_{K}:\widehat{K}\to K$, with non-singular Jacobian, where $\widehat{K}$ is the $d$-dimensional reference simplex or cube. The mesh is allowed to contain a uniformly fixed number of regular hanging nodes per face. Given two meshes $\mathcal{T}_1$ and $\mathcal{T}_2$, we denote their \emph{coarsest common refinement} by $\mathcal{T}_1 \vee \mathcal{T}_2$.

With these definitions, the finite element space $\mathbb{V}_{\cG}(\mathcal{T})$ over the mesh $\mathcal{T}$ is given by
\begin{equation}\label{eq:FEspace}
\mathbb{V}_{\cG}(\mathcal{T}) := \{v \in  \H^1_0(\Omega) : v|_K\circ F_K \in \mathbb{S}^p(\widehat{K}), \, K \in \mathcal{T} \},
\end{equation}
where $\displaystyle\mathbb{S}^p(\widehat{K})$ denotes the space $\displaystyle\P^p(\widehat K)$ of all polynomials of total degree $p$ if $\widehat{K}$ is the $d$-dimensional reference simplex, or the space $\displaystyle\mathbb{Q}^p(\widehat K)$ of all polynomials of degree $p$ in each variable if $\widehat{K}$ is the $d$-dimensional  reference cube. Here, $\displaystyle\H^1_0(\Omega)$ is the Sobolev space of functions with weak gradient in $\Ltwo(\Omega)$ and zero trace on $\partial\Omega$.

For a given function $g \in \C^0({\Omega})$,
suppose that the elliptic problem
\begin{equation}\label{eq:ell}
\begin{aligned}
-\diff\Delta{w} &= g \quad\text{in }\Omega,\qquad
w =0 \quad\text{on }  \partial\Omega,
\end{aligned}
\end{equation}
has a unique solution $\displaystyle w \in \H^1_0(\Omega)\!\cap\! \C^0(\overline{\Omega})$. The conforming finite element approximation of $w$, $w_{\cG}\in\mathbb{V}_{\cG}(\mathcal{T})$, is given by the solution of
\begin{equation}\label{eq:FEM}
\diff(\nabla w_{\cG},\nabla v_{\cG})=(g,v_{\cG})\qquad\forall v_{\cG}\in\mathbb{V}_{\cG}(\mathcal{T}),
\end{equation}
where $(\cdot,\cdot)$ signifies the $\Ltwo(\Omega)$-inner product. In order to make the forthcoming error analysis in \S\ref{sc:erroranalysis} as general as possible, we make the assumption that an error bound for the elliptic problem \eqref{eq:ell} and its finite element counterpart \eqref{eq:FEM} exists in the $\Linf(\Omega)$-norm $\|\cdot\|$ but refrain from specifying the estimator.

\begin{assumption}\label{as:ell}
There exists an a posteriori error estimator $\mathcal{E}$ for the error between the solution of~\eqref{eq:ell} and its conforming FEM approximation~\eqref{eq:FEM} such that the following pointwise bound holds
\begin{equation}
\begin{aligned}
\notag
\|w-w_{\cG}\| \leq C_{\infty}\mathcal{E}(w_{\cG},g,\mathcal{T}),
\end{aligned}
\end{equation}
where $C_{\infty}>0$ is a constant that is independent of the mesh-size, $\diff$, $w$ and $w_{\cG}$.
\end{assumption}
To the best of our knowledge, the sharpest elliptic \emph{a posteriori} error estimators currently available in the literature for the $\Linf$ norm are from \cite{L14} for $d=1$ and from \cite{DK14} for $d = 2$ or $d=3$; we would currently recommend that these estimators are used for $\mathcal{E}$. Other elliptic \emph{a posteriori} error estimators for the $\Linf$ norm are also available \cite{L07, N06, K15} which could be useful in certain situations.

\subsection{Fully Discrete Scheme}

For the temporal discontinuous Galerkin discretisation, we introduce a sequence of time nodes $0 := t_0 < t_1 < \ldots$ which define a time partition $\M:=\{I_m\}_{m\ge1}$ in terms of open time intervals~$I_m:=(t_{m-1},t_m)$, $m\ge 1$. The length $k_m := t_m - t_{m-1}$ (which may be variable) of the time interval $I_m$ is called the time step length.

Let $\mathcal{T}_1$ denote an initial spatial mesh of $\Omega$ associated with the first time interval $I_1$. Then, to each successive interval $I_m$, $m\ge 2$, we associate a spatial mesh $\mathcal{T}_m$ which is assumed to have been obtained from $\mathcal{T}_{m-1}$ by local refinement and/or coarsening. This restriction upon mesh change is made in order to avoid degradation of the finite element solution, cf. \cite{BKM13,D82}. To each interval $I_m$, we then assign the finite element space $\mathbb{V}_{\cG}^m := \mathbb{V}_{\cG}(\mathcal{T}_m)$, cf.~\eqref{eq:FEspace}.

Additionally, to each interval $I_m$, we associate the (possibly variable) polynomial degree $r_m \in \mathbb{N}\cup\{0\}$ which takes the role of a local approximation order. Given a (real) vector space~$\mathbb{X}$ and some~$r\in\mathbb{N}\!\cup\!\{0\}$, the set
\[
\mathbb{P}^{r}(J;\mathbb{X}):=\left\{\mathrm{p}\in \C^0(J;\mathbb{X}):\,\mathrm{p}(t)=\sum_{i=0}^rx_it^i,\, x_i\in \mathbb{X}\right\}\!,
\]
signifies the space of all polynomials of degree at most~$r$ on an interval~$J\subset\mathbb{R}$ with values in~$\mathbb{X}$.


Finally, given a piecewise-in-time continuous function $U$ which may be discontinuous at each time node $t_m$, we signify by $\timejump{U}_m := U(t_m^+) - U(t_m^-)$ the temporal jump of $U$ with the one-sided limits $\displaystyle U(t_m^\pm) \coloneqq \lim_{s\searrow 0} U(t_m\pm s)$. For $m=0$, we define the initial value 
\begin{equation}\label{eq:init}
U(t^{-}_0) := \pi_1 u_0,
\end{equation}
with $u_0$ from~\eqref{model_strong} where $\pi_1$ is a suitable projection or interpolation operator into the finite element space $\mathbb{V}_{\cG}^1$. 

\begin{remark}\label{rem:init}
We point out that the choice of projection or interpolation for the initial condition $u_0$ has a large impact on the effectiveness of the \emph{a posteriori} error estimator derived in \S\ref{sc:erroranalysis}. In particular, problems were observed when the $\Ltwo(\Omega)$-projection or the standard nodal interpolant were used. It transpires that a good choice is to use the \emph{energy projection} $\pi_1 u_0 \in \mathbb{V}_{\cG}^1$ given by the solution of the finite element problem 
\begin{equation}\label{eq:pi1}
(\nabla \pi_1 u_0, \nabla v) = (- \Delta u_0, v) \qquad \forall v \in \mathbb{V}_{\cG}^1.
\end{equation}
\end{remark}

With the above notation at hand, the dG-cG discretisation of \eqref{model_strong} reads as follows: For $m\ge 1$, we seek $U |_{I_m} \in \P^{r_m}(I_m; \mathbb{V}_{\cG}^m)$ such that
\begin{multline}\label{dg_method}
\int_{I_m} (U_t,\,V) \dif s+ (\timejump{U}_{m-1},\,V(t_{m-1}^+)) + \int_{I_m} \diff(\nabla U, \,\nabla V) \dif s\\ = \int_{I_m} (f(U),\,V) \dif s 
\qquad\forall V \in \P^{r_m}(I_m; \mathbb{V}_{\cG}^m),
\end{multline}
cf., e.g., \cite{GLW17} for the linear case when $f$ is independent of $U$.

\section{Reconstructions}\label{sc:rec}

Following the approach pioneered in \cite{GLW17}, the \emph{a posteriori} error analysis in the subsequent section~\S\ref{sc:erroranalysis} will be based on suitable temporal and spatial reconstructions.

\subsection{Temporal Reconstruction}

Denote by 
\[
\Pi_m : \Ltwo(I_m;\Ltwo(\Omega)) \rightarrow \P^{r_m}(I_m; \Ltwo(\Omega))
\]
the \emph{temporal $\Ltwo$-projection operator} satisfying
\begin{equation}\label{eq:L2}
\begin{aligned}
f\mapsto \Pi_m(f),\qquad \int_{I_m} (\Pi_m(f),v) \dif s = \int_{I_m} (f,v) \dif s \qquad \forall v \in \P^{r_m}(I_m; \Ltwo(\Omega)).
\end{aligned}
\end{equation}
Furthermore, following the variable order approach~\cite{GLW17,SW19,HW18}, we introduce the \emph{temporal lifting operator} $\chi_m: \mathbb{X} \rightarrow \P^{r_m}(I_m; \mathbb{X})$ (where $\mathbb{X}\subset\Ltwo(\Omega)$ is a linear subspace) originally introduced in~\cite{MN06} which is defined implicitly by
\begin{equation}
\begin{aligned}
\label{lift}
z\mapsto \chi_m(z),\qquad \int_{I_m} (\chi_m(z),v) \dif s = (z,v(t_{m-1}^+)) \qquad \forall v \in \P^{r_m}(I_m; \mathbb{X}). 
\end{aligned}
\end{equation}
The lifting operator $\chi_m$ has an explicit representation which is the subject of the ensuing lemma, see~\cite{ScWi10,HW18} for details.

\begin{lemma}
\label{liftlemma}
Let $\widehat I = [-1,1]$ and $F_m:\widehat{I} \to I_m$ be the affine transformation given by $F_m(\hat t):=(k_m\hat{t} + (t_{m-1}+t_m))/2$, $\hat t\in\widehat I$. Furthermore, let $\widehat{\chi}_m$ denote the lifting operator of \eqref{lift} but taken over the reference interval $\widehat I$ instead of $I_m$ then we have
\[
z - \int_{t_{m-1}}^t \chi_m(z) \dif s = z - \int_{-1}^{F_m^{-1}(t)} \widehat{\chi}_m(z) \dif s = -Q_m(t)z, \qquad z \in \mathbb{X}, \quad t \in I_m.
\]
Here,
\begin{equation}\label{eq:Q}
Q_m(t) := \frac12(-1)^{r_m}\!\left(\widehat{\mathrm{L}}_{r_m+1}(F_m^{-1}(t))-\widehat{\mathrm{L}}_{r_m}(F_m^{-1}(t)) \right),\qquad t\in I_m,
\end{equation}
where $\widehat{\mathrm{L}}_i$ denotes the $i^\text{th}$ Legendre polynomial on the interval $\widehat I$. Additionally, by differentiating the above, we can also obtain an explicit formulation of $\chi_m$: 
\[
\chi_m(z)(t) = Q_m'(t)z, \qquad z \in \mathbb{X}, \quad t \in I_m.
\]
\end{lemma}

The \emph{temporal reconstruction} of the solution $U$ to ~\eqref{dg_method} is defined by
\begin{equation}\label{eq:Urec}
\widetilde{U}(t) := U(t) +Q_m(t)\timejump{U}_{m-1},\qquad t \in {I}_m, \quad m\ge 1,
\end{equation}
with $Q_m$ from~\eqref{eq:Q}. We thus have
$
\widetilde{U}_t=U_t+\chi_m(\timejump{U}_{m-1})
$
on $I_m$ and so
\[
\int_{I_m}(\widetilde{U}_t,V)\dif s =\int_{I_m} (U_t,\,V) \dif s+ (\timejump{U}_{m-1},\,V(t_{m-1}^+)) \qquad \forall V \in \P^{r_m}(I_m; \mathbb{V}_{\cG}^m).
\]
Invoking~\eqref{eq:L2}, the dG-cG method~\eqref{dg_method} can be equivalently written as
\begin{equation}\label{eq:dGrec1}
\int_{I_m} (\widetilde U_t,\,V) \dif s+ \diff\int_{I_m} (\nabla U, \,\nabla V) \dif s = \int_{I_m} (\Pi_m (f(U)),\,V) \dif s,
\end{equation}
for any $V \in \P^{r_m}(I_m; \mathbb{V}_{\cG}^m)$.

\subsection{Spatial Reconstruction}

We recall the elliptic reconstruction technique \cite{LM06,MN03}. To that end, we introduce the \emph{discrete laplacian} given by
\begin{equation}\label{eq:Am}
\A |_{I_m} := \Pi_m(f(U))-\widetilde U_t|_{I_m}. 
\end{equation}
Then, from~\eqref{eq:dGrec1}, we observe that
\[
\diff\int_{I_m} (\nabla U, \,\nabla V) \dif s = \int_{I_m} (\A ,\,V) \dif s 
\qquad\forall V \in \P^{r_m}(I_m; \mathbb{V}_{\cG}^m).
\]
Now, for $t\in I_m$, define the \emph{elliptic reconstruction} $\omega(t)\in \H^1_0(\Omega)$ to be the unique solution of the elliptic problem
\begin{equation}
\begin{aligned}
\label{elliprecon}
\diff(\nabla \omega(t), \nabla v) =(\A(t),v)  \qquad \forall v\in \H^1_0(\Omega).
\end{aligned}
\end{equation}
In strong form this reads
\begin{equation}\label{eq:wstrong}
-\diff\Delta\omega(t)=\A(t)\qquad\text{in }\H^{-1}(\Omega),
\end{equation}
with $\H^{-1}(\Omega)$ the dual space of $\H^1_0(\Omega)$. Upon expanding $\omega$ and $\A$ in an (orthonormal) Legendre basis in time, we see that $\omega|_{I_m}\in\P^{r_m}(I_m;\H^1_0(\Omega))$. Moreover, we have
\[
\diff\int_{I_m}(\nabla (\omega-U), \nabla V) \dif s= 0 
\qquad\forall V\in\P^{r_m}(I_m;\mathbb{V}_{\cG}^m).
\]
Since $(\omega-U) |_{I_m} \in\P^{r_m}(I_m;\H^1_0(\Omega))$, we deduce that
\[
\diff(\nabla (\omega-U)(t), \nabla v_{\cG})=0\qquad\forall v_{\cG}\in\mathbb{V}_{\cG}^m\quad \forall t\in I_m,
\]
i.e., $U |_{I_m}$ is the (pointwise) finite element approximation of $\omega |_{I_m}$ and so $\omega - U$ is an \emph{elliptic error} that can be bounded with the aid of the \emph{a posteriori} error estimator from Assumption~\ref{as:ell}.

\begin{remark}
If $\pi_1$ from \eqref{eq:init} is the energy projection proposed in~\eqref{eq:pi1} then
\[
\diff(\nabla U(t_0^-),\nabla v_{\cG})=\diff(\nabla(\pi_1u_0),\nabla v_{\cG})=\diff(\nabla u_0,\nabla v_{\cG})\qquad\forall v_{\cG}\in\mathbb{V}_{\cG}^1.
\]
By setting
\begin{equation}\label{eq:winit}
\omega(t_0^-) := u_0,
\end{equation}
the initial elliptic error $\omega(t_0^-) - U(t_0^-)$, although already computable, may also be estimated via an elliptic \emph{a posteriori} error estimator as proposed in Assumption~\ref{as:ell}.
\end{remark}

\subsection{Space-Time Reconstruction}

We emphasise that the derivation of \emph{a posteriori} error bounds for dG-in-time discretisations mandates the reconstruction of both $U$ and $\omega$, otherwise, suboptimal error estimates in time will result; cf.~\cite{AkMaNo09}. For this reason, we introduce the temporal reconstruction $\widetilde{\omega}$ of $\omega$ from~\eqref{elliprecon} given by
\begin{equation}
\label{timerecon}
\widetilde{\omega}(t) := \omega(t) +Q_m(t)\timejump{\omega}_{m-1}, \qquad t \in {I}_m, \quad m \ge 1,
\end{equation}
While $U$ \eqref{dg_method} and $\omega$ \eqref{elliprecon} are only piecewise continuous with a jump discontinuity at each time node, we have the following result for their reconstructions.

\begin{lemma}\label{lem:rec}
The reconstructions $\widetilde{U}$ \eqref{eq:Urec} and $\widetilde{\omega}$ \eqref{timerecon} are both continuous. 
\end{lemma}

\begin{proof}
Let $m\ge 2$. Using the fact that Legendre polynomials on the interval $[-1,1]$ satisfy $\widehat{\mathrm{L}}_i(\pm1)=(\pm1)^i$, we have $Q_m(t_{m-1}^+)=-1$ and $Q_{m-1}(t_{m-1}^-)=0$ hence
\begin{align*}
\widetilde U(t_{m-1}^+)
=U(t_{m-1}^+)+Q_m(t_{m-1}^+)\timejump{U}_m
=U(t_{m-1}^-)
=\widetilde U(t_{m-1}^-).
\end{align*}
Therefore, $\widetilde U$ is continuous at $t_{m-1}$. The continuity of $\widetilde\omega$ is similarly established.
\end{proof}

Furthermore, recalling~\eqref{eq:Am}, we introduce the temporal reconstruction $\widetilde{\A}$ of the discrete laplacian $\A$ which is similarly given by 
\begin{equation}\label{Atildedef}
\widetilde{\A}(t):=\A(t) + Q_m(t)\timejump{\A}_{m-1},\qquad t\in I_m, \quad m \geq 1,
\end{equation}
where we set $\A(t_0^-) := -\diff \Delta u_0$ in light of \eqref{eq:wstrong}. It thus follows from \eqref{eq:wstrong} that
\begin{equation}\label{eq:w'}
-\!\diff\Delta\widetilde\omega=\widetilde{\A}\qquad\text{in }\H^{-1}(\Omega).
\end{equation}

\section{A Posteriori Error Analysis}\label{sc:erroranalysis}

\subsection{Temporal Error}

For the error analysis, we proceed roughly along the lines of \cite{IM17} by constructing an error equation for the \emph{reconstructed parabolic error} $\widetilde{\rho}:= u - \widetilde{\omega}$. Note that $\widetilde{\rho}$ is continuous due to Lemma~\ref{lem:rec}. We begin by subtracting \eqref{eq:w'} from \eqref{model_strong} which gives
\[
u_t -\diff\Delta \widetilde{\rho} = f(u) - \widetilde{\A}\qquad\text{in }\H^{-1}(\Omega).
\]
Next, we add and subtract $f(\widetilde{U})$ and $\widetilde{U}_t$ yielding the \emph{error equation}
\begin{equation}
\begin{aligned}
\label{erroreqn}
\widetilde{\rho}_t -\diff\Delta \widetilde{\rho} = f(u) - f(\widetilde{U}) + R_{\mathrm{time}} - \widetilde{\epsilon}_t \qquad\text{in }\H^{-1}(\Omega),
\end{aligned}
\end{equation}
where $R_{\mathrm{time}}$ is the \emph{temporal residual} given by $R_{\mathrm{time}}:= f(\widetilde{U}) - \widetilde{U}_t - \widetilde{\A}$ and $\widetilde{\epsilon} := \widetilde{\omega} - \widetilde{U}$ is the \emph{reconstructed elliptic error}. Letting $\widetilde{e} := u - \widetilde{U}$ denote the \emph{reconstructed error} associated with the temporal reconstruction $\widetilde{U}$ then we have the decomposition $\widetilde{e} = \widetilde{\rho} + \widetilde{\epsilon}$. It can be seen that $R_{\mathrm{time}}$ is optimal order in time by substituting $\widetilde{\A}$ \eqref{Atildedef}.

By using standard parabolic semi-group theory, see, e.g., \cite[p 93]{T84}, it can be inferred that the solution $\widetilde{\rho}$ of the error equation \eqref{erroreqn}, for each $m\ge 1$ for which $\widetilde{\rho}$ exists on $I_m$, can be represented formally by
\[
\widetilde\rho(t)=e^{-(t-t_{m-1})\diff\Delta}\widetilde\rho(t_{m-1})
+\int_{t_{m-1}}^te^{-(t-s)\diff\Delta}[f(u) - f(\widetilde{U}) + R_{\mathrm{time}} - \widetilde{\epsilon}_t] \dif s,
\]
$t\in I_m$. Equivalently, upon noticing that $u=\widetilde\rho+\widetilde\omega$, 
we devise the fixed point equation $\Phi_m(\widetilde\rho)=\widetilde\rho$ where
\[
\Phi_m(v)(t)  := \mathrm{e}^{-(t-t_{m-1}) \diff\Delta} \widetilde{\rho}(t_{m-1}) + \int_{t_{m-1}}^t\mathrm{e}^{-(t-s) \diff\Delta} [f(v+\widetilde{\omega})-f(\widetilde{U})+R_{\mathrm{time}}-\widetilde{\epsilon}_t ] \dif s,
\]
$t\in I_m$. 

For $m\ge 1$, we define the local $\Linf(\Linf)$ norm
\begin{equation}
\begin{aligned}
\notag
\|v\|_m := \sup_{t \in I_m} \|v(t)\|,\qquad v\in\C^0(\overline{I}_m;\Linf(\Omega)).
\end{aligned}
\end{equation}
Then, employing the Banach fixed point theorem, our goal is now to show that $\Phi_m$ has a unique fixed point in some ball centered on zero in the $\|\cdot\|_m$ norm which, by Duhamel's Principle, must also solve the error equation \eqref{erroreqn}. The fundamental idea here is that we will construct the radius of the ball to be a computable quantity which, in turn, implies an \emph{a posteriori} error bound for $\|\widetilde{\rho}\|_m$. Before doing so, we require some general \emph{a posteriori} error bounds to hold for the elliptic error $\widetilde{\epsilon}$. To that end, we make the following assumptions.

\begin{assumption}
\label{spaceest}
We assume that the estimates
\begin{align*}
\|\widetilde{\epsilon}(t)\| & \leq C_{\infty}\eta_{\mathrm{space}}(t) :=  C_{\infty} \mathcal{E}(\widetilde{U}(t), \,\widetilde{\A}(t), \,\mathcal{T}_{m-2} \vee\mathcal{T}_{m-1} \vee \mathcal{T}_{m}),  && t \in I_m,\,\, m \ge 1,\\
\|\widetilde{\epsilon}_t(t)\| & \leq C_{\infty}{\accentset{\bigcdot}{\eta}}_{\mathrm{space}}(t) := C_{\infty} \mathcal{E}(\widetilde{U}_t(t), \,\widetilde{\A}_t(t),\, \mathcal{T}_{m-2} \vee\mathcal{T}_{m-1} \vee \mathcal{T}_{m}), && t \in I_m,\,\, m \ge 1,
\end{align*}
hold where $\eta_{\mathrm{space}}$ is the space estimator and ${\accentset{\bigcdot}{\eta}}_{\mathrm{space}}$ is the space derivative estimator. Here, $C_{\infty}>0$ is a constant that is independent of the mesh-size, $\diff$, $u$ and $U$ but may be dependent upon the number of refinement levels between the meshes $\mathcal{T}_{m-2}$, $\mathcal{T}_{m-1}$ and $\mathcal{T}_{m}$ (for $m\ge 3$).
\end{assumption}

\begin{remark}
Assumption \ref{spaceest} can be shown to hold for residual-based elliptic \emph{a posteriori} error estimators including any of those from \cite{L07, N06, K15} as well as our recommended choices \cite{L14, DK14}. Indeed, since $\widetilde{\omega}$ enjoys the elliptic reconstruction property (see above) and $\widetilde{\A}(t) \in \C^0(\Omega)$ for any $t \in I_m$, the elliptic error $\|\widetilde{\epsilon}(t)\|$ can be estimated using Assumption \ref{as:ell} thereby yielding the first bound of Assumption  \ref{spaceest}. The second bound of Assumption \ref{spaceest} follows from Assumption \ref{as:ell} upon noting that $\widetilde{\omega}_t$ also possesses the elliptic reconstruction property provided that the discrete Laplacian satisfies $\widetilde{\A}_t(t) \in \C^0(\Omega)$ which is true, for instance, under the assumption that the nonlinearity $f$ is continuously differentiable.
\end{remark}

In order to proceed with our analysis, we first suppose that $\psi_m$ is a computable quantity such that
\begin{equation}
\label{psidef1}
\|\widetilde{\rho}(t_{m-1})\| +  \int_{I_m} \!\!\mathfrak{L}(s,\|\widetilde{U}(s)\|,\|\widetilde{U}(s)\|+\|\widetilde{\epsilon}(s)\|)\|\widetilde{\epsilon}(s)\| \dif s+  \eta_{\mathrm{time}}^m + \int_{I_m}\!\! \|\widetilde{\epsilon}_t(s)\| \dif s  \leq \psi_m,
\end{equation}
holds where $\displaystyle\eta_{\mathrm{time}}^m$ is the \emph{time estimator} given by 
\begin{equation}
\begin{aligned}
\notag
\eta_{\mathrm{time}}^m := \int_{I_m} \|R_{\mathrm{time}}(s)\| \dif s.
\end{aligned}
\end{equation}
Furthermore, if it exists, let $\delta_m \in [1,\infty)$ denote the smallest root of the function $\phi_m: [1, \infty) \to \mathbb{R}$ defined by 
\begin{equation}\label{phi}
\phi_m(\delta) := 1+ \delta\left[\int_{I_m} L(s, \delta) \dif s -1 \right],
\end{equation}
where
\begin{equation}
\begin{aligned}
\notag
L(s,\delta) := \mathfrak{L}(s,\delta\psi_m + \|\widetilde{U}(s)\| + C_{\infty}\eta_{\mathrm{space}}(s), \delta\psi_m + \|\widetilde{U}(s)\| +  C_{\infty}\eta_{\mathrm{space}}(s)),
\end{aligned}
\end{equation}
for $s \in I_m$ and $\delta \in [1,\infty)$.

We now have the following preliminary error bound.

\begin{proposition}\label{pr:rhoerrorbound}
If the time step length $k_m$, $m\ge 1$, is chosen sufficiently small then $\delta_m \in [1,\infty)$ exists and the error equation \eqref{erroreqn} has a unique local solution $\widetilde{\rho}$ that satisfies the error bound
$
\|\widetilde{\rho}\|_m \leq \theta_m\psi_m
$
where $\theta_m \in [1,\infty)$ is given by 
\begin{equation}\label{eq:theta}
\theta_m := \exp\!\bigg(\int_{I_m}\!\!\mathfrak{L}(s,\delta_m\psi_m+ \|\widetilde{U}(s)\| + C_{\infty}\eta_{\mathrm{space}}(s),\|\widetilde{U}(s)\| + C_{\infty}\eta_{\mathrm{space}}(s)) \dif s \bigg) \geq 1.
\end{equation}
\end{proposition}
\begin{proof}
The error equation \eqref{erroreqn} has precisely the same form as the error equation from \cite{IM17} which, under the nonlinearity property \eqref{eq:Lip} and upon assuming \eqref{psidef1}, was shown to have a unique local solution satisfying the above error bound provided that $\delta_m$ exists. Additionally, it was shown in \cite{KMW16,IM17} that $\phi_m$ has a root $\delta_m \in [1,\infty)$ if $k_m$ is chosen sufficiently small.
\end{proof}

\subsection{Computable Error Bound}

In order to transform the parabolic reconstruction error bound of Proposition \ref{pr:rhoerrorbound} into an \emph{a posteriori} error estimate, all that is required is an explicit characterization of $\psi_m$; specifically, we need to bound the various terms on the left-hand side of \eqref{psidef1} by computable quantities.

To begin, we note that Proposition \ref{pr:rhoerrorbound} implies that if $\delta_1, \ldots, \delta_{m-1}$ exist then we have
$
\|\widetilde{\rho}(t_{m-1})\| \leq \|\widetilde{\rho}\|_{m-1} \leq \theta_{m-1}\psi_{m-1}
$
for $m\ge 2$. For the first interval, we instead have the explicit relation $\widetilde{\rho}(t_0) = u(t_0) - \widetilde{\omega}(t_0) = u_0 - u_0 = 0$, cf.~\eqref{eq:winit}. To bound the first term containing $\widetilde{\epsilon}$ in \eqref{psidef1}, we utilise Assumption \ref{spaceest} together with the monotonicity property of $\mathfrak{L}$, viz.,
\begin{multline*}
\int_{I_m} \mathfrak{L}(s,\|\widetilde{U}(s)\|,\|\widetilde{U}(s)\|+\|\widetilde{\epsilon}(s)\|)\,\|\widetilde{\epsilon}(s)\| \dif s\\ \leq C_{\infty}\int_{I_m} \mathfrak{L}(s,\|\widetilde{U}(s)\|,\|\widetilde{U}(s)\|+C_{\infty}\eta_{\mathrm{space}}(s))\,\eta_{\mathrm{space}}(s) \dif s.
\end{multline*}
The second term containing $\widetilde{\epsilon}$ in \eqref{psidef1} can be estimated directly by applying Assumption \ref{spaceest} as follows
\[
 \int_{I_m} \|\widetilde{\epsilon}_t(s)\| \dif s  \leq  C_{\infty}\int_{I_m} {\accentset{\bigcdot}{\eta}}_{\mathrm{space}}(s) \dif s.
\]
Combining the estimates, we see that we can define $\psi_m$ from~\eqref{psidef1} recursively by
\begin{multline}\label{psi}
\psi_m  := \theta_{m-1}\psi_{m-1}+ C_{\infty}\int_{I_m} \mathfrak{L}(s,\|\widetilde{U}(s)\|,\|\widetilde{U}(s)\|+C_{\infty}\eta_{\mathrm{space}}(s))\,\eta_{\mathrm{space}}(s) \dif s\\+\eta_{\mathrm{time}}^m+C_{\infty}\int_{I_m} {\accentset{\bigcdot}{\eta}}_{\mathrm{space}}(s) \dif s,
\end{multline}
with $\theta_0 := 1$ and $\psi_0 := 0$. With $\psi_m$ defined, all components of the error bound are now in place as well as fully computable, and we are ready to state the main result.

\begin{theorem}
\label{maintheorem}
For $M\ge 1$, suppose that $\delta_1, \ldots, \delta_M$ exist then the reconstructed error of the dG-cG method \eqref{dg_method} satisfies the $\Linf(\Linf)$ a posteriori bound
\begin{equation}
\notag
\max_{1 \leq m \leq M} \|\widetilde{e}\|_m \leq \theta_M \psi_M + C_{\infty}\|\eta_{\mathrm{space}}\|_{\Linf(0,t_M)}.
\end{equation}
\end{theorem}
\begin{proof}
Since $\delta_1,\ldots,\delta_M$ exist, then from Proposition \ref{pr:rhoerrorbound} we have the error bound
$\|\widetilde{\rho}\|_m \leq \theta_m\psi_m$
for any $1 \leq m \leq M$. Noting~\eqref{eq:theta} and~\eqref{psi}, we observe that the sequence $\{\theta_m\psi_m\}_{m\ge 1}$ is monotone increasing. Therefore, upon recalling the decomposition $\widetilde{e} = \widetilde{\rho} + \widetilde{\epsilon}$ we obtain
\begin{equation}
\notag
\max_{1 \leq m \leq M} \|\widetilde{e}\|_m \leq \max_{1 \leq m \leq M} \|\widetilde{\rho}\|_m + \max_{1 \leq m \leq M} \|\widetilde{\epsilon}\|_m \leq \theta_M \psi_M + \max_{1 \leq m \leq M} \|\widetilde{\epsilon}\|_m.
\end{equation}
The stated result then follows from Assumption \ref{spaceest}.
\end{proof}

In some sense, the temporal reconstruction $\widetilde{U}$ is a better approximation to $u$ than $U$, and so from a practical standpoint, it is often best to use Theorem \ref{maintheorem} directly. For some applications, however, a bound on the error rather than on the reconstructed error may be necessary. Such a bound follows immediately upon rewriting the error, viz., $e := u - U = \widetilde{e} + \widetilde{U} - U$, applying the triangle inequality, and exploiting the reconstructed error bound of Theorem \ref{maintheorem}.

\begin{corollary}
Suppose that $\delta_1, \ldots, \delta_M$ exist then the error of the dG-cG method \eqref{dg_method} satisfies the $\Linf(\Linf)$ a posteriori bound
\begin{equation}
\notag
\max_{1 \leq m \leq M} \|{e}\|_m \leq \theta_M \psi_M + C_{\infty}\|\eta_{\mathrm{space}}\|_{\Linf(0,t_M)} + \max_{1 \leq m \leq M} \|U - \widetilde{U}\|_m.
\end{equation}
\end{corollary}

\section{Adaptive Algorithm}\label{sc:adaptive}

In order to determine when it is necessary to change the mesh and/or time step size, we need to define \emph{refinement indicators} based on Theorem \ref{maintheorem}. Asymptotically, we expect that only the terms which contribute to $\theta_M \psi_M$ are likely to play a significant role in the magnitude of the error bound. We therefore define our refinement indicators based only on the term $\psi_m$ from~\eqref{psi}. In order to characterise the new refinement indicators, we define the product
$
\widetilde{\theta}_m := \prod_{i = 0}^m \theta_i, \,m \ge 0.
$
For a blow-up problem, $\widetilde{\theta}_m$ plays the role of approximating the rate of blow-up, cf. \cite{KMW16}. It is clear that we should use the time estimator, $\eta_{\mathrm{time}}^m$, in order to select the time step size. To this end, if the goal is to minimise the estimator in Theorem \ref{maintheorem} then the scaling nature of the estimator suggests that we demand that
\[
\mathrm{ref}^m_{\mathrm{time}} := \widetilde{\theta}_m^{-1}\eta_{\mathrm{time}}^m \leq {\tt ttol},\qquad m\ge 1,
\]
is satisfied where $\text{\tt ttol}>0$ is the \emph{temporal refinement threshold}. This choice is equivalent to using $\eta_{\mathrm{time}}^m$ for the temporal refinement indicator while increasing the temporal refinement threshold by $\theta_m$ after each time step. Practically speaking, this condition means that the indicators are allowed to increase in proportion to how close we are to the blow-up time. One can imagine that if we did not enforce such a condition that the time step lengths and spatial mesh sizes would become unpractically small as we approach the blow-up time. Given that we are dealing with blow-up problems which are unlikely to necessitate larger time step sizes as we approach the blow-up time, we do not see the need to introduce temporal coarsening.

The spatial terms in $\psi_m$ are based on the space estimators in Assumption \ref{spaceest} which are evaluated on the union of the current mesh $\mathcal{T}_m$ with the previous two meshes $\mathcal{T}_{m-1}$ and $\mathcal{T}_{m-2}$.  Given that each element $K \in \mathcal{T}_m$ is either contained in the union mesh directly or has subelements $\widecheck{K}$ contained in the union mesh, and that our error estimator is based on the $\Linf$ norm, it is natural to set the spatial refinement indicator to
\[
\mathrm{ref}^m_{\mathrm{space}} \big|_K := \widetilde{\theta}_m^{-1}k_m^{-1}\max_{\widecheck{K} \subset K} \Lambda^m_{\widecheck{K}},
\]
where
\[
\Lambda^m_{\widecheck{K}}:=
\int_{I_m} \mathfrak{L}(s,\|\widetilde{U}(s)\|,\|\widetilde{U}(s)\|+\eta_{\mathrm{space}}(s))\,\eta_{\mathrm{space}}(s)\big|_{\widecheck{K}} \dif s+\int_{I_m} {\accentset{\bigcdot}{\eta}}_{\mathrm{space}}(s)\big|_{\widecheck{K}} \dif s.
\]
Note that we divide by the time step length to normalize the spatial refinement indicator in order to ensure that it is temporally independent. We then demand that
\[
{\tt stol^-} \leq \mathrm{ref}^m_{\mathrm{space}} \big|_K  \leq {\tt stol^+}
\]
is satisfied for all $K \in \mathcal{T}_m$. Here, $\textrm{\tt stol$^+$}>0$ is the \emph{spatial refinement threshold} and $\textrm{\tt stol$^-$}>0$ is the \emph{spatial coarsening threshold}. Based on expected rates of convergence, we set ${\tt stol}^- = 0.1\times2^{-p}\times{\tt stol^+}$  in our numerical experiments where $p$ is the spatial polynomial degree in~\eqref{eq:FEspace}. For the elliptic error estimator $\mathcal{E}$ in Assumption \ref{as:ell} which forms the basis of the space estimators $\eta_{\mathrm{space}}$ and ${\accentset{\bigcdot}{\eta}}_{\mathrm{space}}$, we use \cite{L14} for $d = 1$ and \cite{DK14} otherwise.

The dG-cG method \eqref{dg_method} is both implicit and possibly nonlinear; as such, it is necessary to use an iterative method to solve the underlying discrete problem. Standard choices include Picard and Newton iterations. Incidentally, Picard iteration was observed to be faster because the system matrix does not need to be rebuilt at each iteration despite the fact that it requires more iterations. In addition, Picard iteration allows for the use of preconditioners built specifically for the dG-cG method, cf., e.g., \cite{S17}.

With the refinement indicators in place, we are now ready to outline our adaptive strategy. Firstly, we need to refine the initial coarse input mesh and time step length. Ideally, we would use the indicators directly, however, a certain amount of spatial resolution of the initial condition is needed for the temporal refinement indicator to give reasonable values. Therefore, we begin by computing the energy projection $\pi_1 u_0$ of $u_0$, cf.~Remark~\ref{rem:init}, and refining the initial mesh $\mathcal{T}_1$ until 
\[
\|u_0 - \pi_1 u_0\| \leq {\tt stol^-}
\]
is satisfied. Next, we continue by recomputing the numerical solution $U |_{I_1}$ and refining the initial mesh $\mathcal{T}_1$ and time step length $k_1$ until $\mathrm{ref}_{\mathrm{time}}^1 < {\tt ttol}$ and $\mathrm{ref}_{\mathrm{space}}^1 \big|_K < {\tt stol^+}$ for all $K \in \mathcal{T}_1$.

For a general time step $m>1$, we first set $\mathcal{T}_m \gets \mathcal{T}_{m-1}$ and $k_m \gets k_{m-1}$. We then compute the numerical solution $U |_{I_m}$ and calculate the refinement indicators $\mathrm{ref}_{\mathrm{space}}^m$ and $\mathrm{ref}_{\mathrm{time}}^m$. If $\mathrm{ref}_{\mathrm{time}}^m > {\tt ttol}$, we set $k_m \gets \nicefrac{k_{m-1}}{2}$. We also refine all elements $K \in \mathcal{T}_m$ with $\mathrm{ref}_{\mathrm{space}}^m \big|_K > {\tt stol^+}$ and coarsen all elements $K \in \mathcal{T}_m$ with $\mathrm{ref}_{\mathrm{space}}^m \big|_K < {\tt stol^-}$. If the mesh or time step length has changed then we recompute the numerical solution $U |_{I_m}$. Once the refinement stage is complete, we recompute the estimators if necessary so that we can evaluate $\psi_m$ in \eqref{psi}. We then attempt to compute $\delta_m\in [1,\infty)$, the smallest root of \eqref{phi}, via a Newton iteration scheme. If the Newton method converges, we continue to the next time step; otherwise, we terminate the adaptive algorithm.

\section{Numerical Experiments}\label{sc:num}

The code for the adaptive algorithm is based on the {\tt deal.II} finite element library version 9.2.0 \cite{DIIv1, BHK07} and may be found via \href{https://github.com/smetcalfephd/Blow-up-in-the-Semilinear-Heat-Equation/}{\color{blue}this link}.

\subsection{Example 1: Point Blowup}
Let $\Omega = (-5,5)^2$, $\diff = 1$, $f(u) = u^2$, and choose the initial condition to be the Gau{\ss}ian blob given by $u_0(x,y) = 10\exp(-2x^2-2y^2)$. The blow-up set for this example consists of only a single point (the origin) which is aligned with the mesh making it spatially simple. For any $v_1, v_2 \in \mathbb{R}$ we have 
\begin{equation*}
\begin{split}
|f(v_1)-f(v_2)| = |v_1^2 - v_2^2 | \leq |v_1 - v_2|(|v_1|+|v_2|).
\end{split}
\end{equation*}
Therefore, we have $\mathfrak{L}(|v_1|,|v_2|) = |v_1| + |v_2|$ in \eqref{eq:Lip} and so $\delta_m$ (if it exists) is the smallest root of the function $\phi_m:[1,\infty) \to \mathbb{R}$ given by
\begin{equation}
\begin{split}
\label{deltaquad}
\phi_m(\delta) &= 1+ \delta\!\left[\int_{I_m} L(s, \delta) \dif s -1 \right]\\
& = 1 + \delta\!\left[2\int_{I_m} \Big[\|\widetilde{U}(s)\| + C_{\infty}\eta_{\mathrm{space}}(s)\Big] \dif s -1  \right]+2k_m\psi_m\delta^2,
\end{split}
\end{equation}
cf.~\eqref{phi}. In this case, we can calculate $\delta_m$ explicitly via the quadratic formula so there is no need to use a root finding algorithm here.

\begin{figure}
\centering
\includegraphics[scale=0.275]{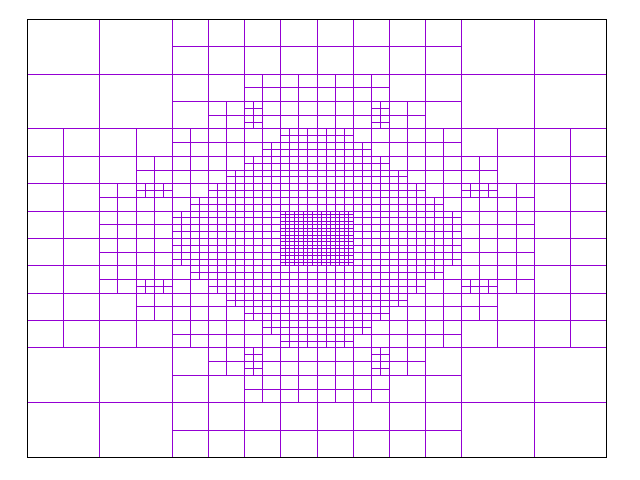} \includegraphics[scale=0.275]{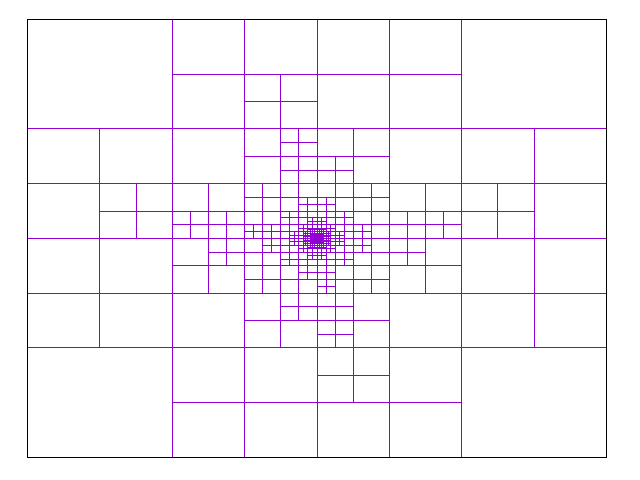}
\\[2ex]
\includegraphics[scale=0.61]{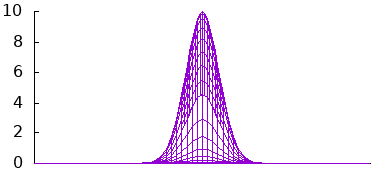}\hfill\includegraphics[scale=0.45]{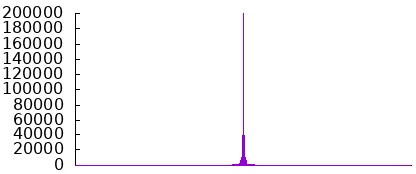}
\caption{Example 1: Initial grid (top left) with $h_{\min} = 0.110485$, final grid (top right) with $h_{\min} = 0.006905$, initial solution profile (bottom left), and final solution profile (bottom right).}
\label{example1solutions}
\end{figure}

With $\delta_m$ defined, we can now apply the adaptive algorithm to this numerical example. We use a finite element space with polynomials of degree $p = 8$ and as small a spatial threshold ${\tt stol}^+$ as is permissible (i.e. as small as can be chosen in the adaptive algorithm until floating point errors start to dominate the computation of the spatial refinement indicator $\mathrm{ref}_{\mathrm{space}}^m$). In time, we allow the polynomial degree ($r = 1, 2, 3$) and the temporal threshold ${\tt ttol}$ to vary in order to see the respective impact on how close we can get to the blow-up time. The results, given in Table \ref{example1results}, clearly illustrate the power of increasing the temporal polynomial degree with $\|U(t^-_N)\|$ being of order $10^3$ for $r = 1$ and of order $10^5$ for $r = 2$ using a similar number of time steps;  likewise, the blow-up time error $|T_{\infty} - t_N|$ shows a similar dramatic effect being of order $10^{-4}$ for $r = 1$ and of order $10^{-6}$ for $r = 2$. For $r = 3$, we note no significant advancement to the blow-up time when compared to $r = 2$ with the primary reason being the inability of the space derivative estimator ${\accentset{\bigcdot}{\eta}}_{\mathrm{space}}$ to be reduced below order $10^{-9}$ owing to floating point errors; by contrast, the time estimator is already several magnitudes smaller being of order $10^{-12}$ in the initial stages of the final computation for $r = 3$. We also include, in Figure \ref{example1solutions}, the spatial meshes and solution profiles at the initial and final times from the last computational run for $r = 3$ which show that the space estimator is choosing to refine the area in the vicinity of the singularity and coarsen elsewhere as expected.

\begin{table}
\begin{center}
\begin{tabular}{cccc} \hline
\toprule
$N$  & 	$\|U(t^-_N)\|$    & 	$t_N$ & $|T_{\infty}-t_N|$\\ \midrule 
4 & 14.1317	& 0.100000 &  1.17e-1 \\ 
31 & 54.8833 & 0.193750 &  2.33e-2\\  
69 &	155.803 & 0.209375 & 7.65e-3 \\  
150 &	632.415 & 0.215234 & 1.79e-3 \\  
315 & 2108.30 & 0.216504 & 5.25e-4\\  
654 &	6856.42 & 0.216870 & 1.59e-4 \\ 
\bottomrule 
\end{tabular}\hfill
\begin{tabular}{cccc} 
\toprule 
$N$  & 	$\|U(t^-_N)\|$    & 	$t_N$ & $|T_{\infty}-t_N|$\\ \midrule 
5 & 16.8504	& 0.125000 & 9.20e-2  \\ 
44 & 522.263 & 0.214844 & 2.19e-3\\ 
87 & 1783.33  & 0.216406 & 6.23e-4  \\ 
152 & 9867.94 & 0.216919 & 1.10e-4  \\ 
296 & 41288.6 & 0.217003 & 2.59e-5  \\ 
609 & 155205 & 0.217022 & 6.83e-6 \\ 
\bottomrule
\end{tabular}\\[3ex]
\begin{tabular}{cccc} \toprule
$N$  & 	$\|U(t^-_N)\|$    & 	$t_N$ & $|T_{\infty}-t_N|$\\ \midrule 
3 & 16.8504 & 0.125000 &   9.20e-2  \\  
22 & 257.817 & 0.212500 & 4.53e-3 \\  
63 & 8091.38 & 0.216895 &  1.34e-4 \\  
156 & 116448 & 0.217020 & 9.12e-6 \\  
261 & 139696 & 0.217021 & 7.59e-6 \\  
451 & 199573 & 0.217023 & 5.30e-6 \\
\bottomrule
\end{tabular}\\[3ex]
\end{center}
\caption{Example 1: Performance data for $r = 1$ (top left), $r = 2$ (top right) and $r = 3$ (bottom).}\label{example1results}
\end{table}

\begin{remark}\label{blow-uprateremark} The results displayed in Table \ref{example1results} utilise $T_{\infty}$ which is not analytically known, however, we can exploit the numerical solution to obtain an approximation. Indeed, it is known that the exact solution $u$ to this blow-up problem satisfies
$$\|u(\cdot, t)\| = C|T_{\infty} - t|^{-\gamma(t)},$$
for some constant $C > 0$ with $\gamma(t) \to 1$ as $t \to T_{\infty}$, cf. \cite{MZ98, MZ00} (``The blow-up rate of the PDE tends to that of the corresponding ODE''). Making the assumption that the numerical solution $U$ also satisfies the above relation implies that 
$$
\frac{\|U(t^-_n)\|}{\|U(t^-_{n-1})\|} = \bigg(\frac{T_{\infty} - t_{n-1}}{T_{\infty} - t_n}\bigg)^{\gamma_n}\qquad \textrm{with }\gamma_n\xrightarrow{n\to\infty}1.
$$
If the two data points are chosen close to the blow-up time then setting $\gamma_n = 1$ should give a good approximation $\widetilde{T}_{\infty}$ to $T_{\infty}$. Rearranging gives the estimate
\[
T_{\infty} \approx \widetilde{T}_{\infty} := \frac{t_n\|U(t^-_n)\| - t_{n-1}\|U(t^-_{n-1})\|}{\|U(t^-_n)\| - \|U(t^-_{n-1})\|},
\]
which yields the estimate $T_{\infty} \approx 0.21702877415$ using the last two data points from the final computational run for $r = 3$. This approximation can also be back-substituted to approximate the numerical blow-up rate $\gamma_n$ via
$$
\gamma_n \approx \frac{\log\|U(t^-_n)\| - \log\|U(t^-_{n-1})\|}{\log(\widetilde{T}_{\infty} - t_{n-1}) - \log(\widetilde{T}_{\infty} - t_n)},
$$
which for points $t_{n-1}$, $t_n$ ``far enough away'' from $\widetilde{T}_{\infty}$ should give a reasonably accurate answer.
\end{remark}

\begin{figure}
\centering
\includegraphics[scale=0.29]{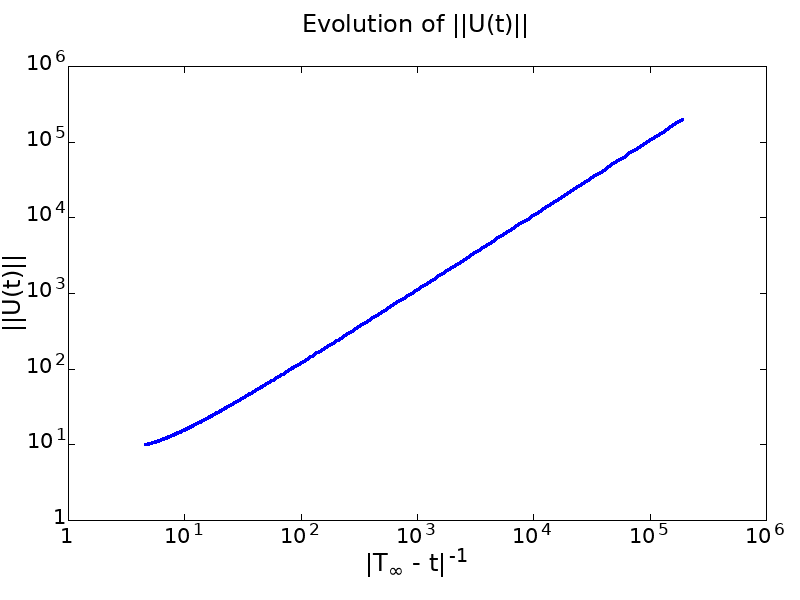} \hfill\includegraphics[scale=0.29]{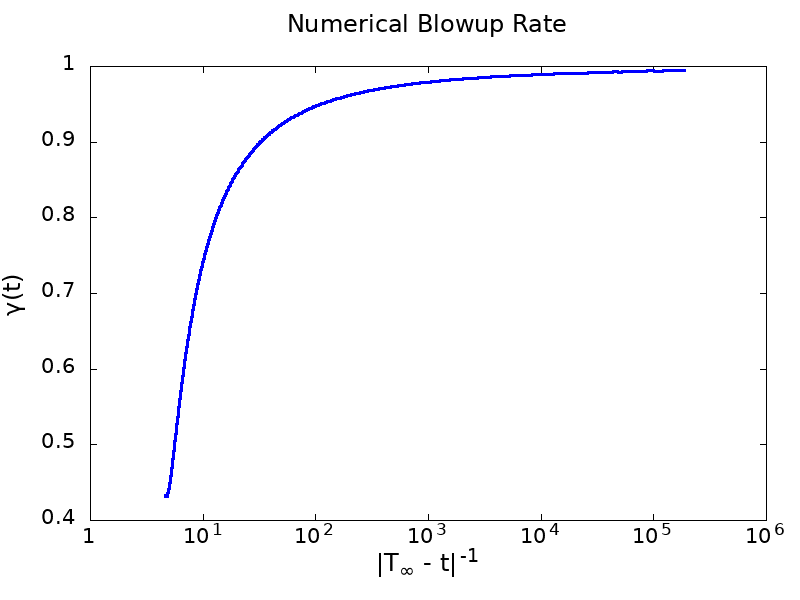}
\\[3ex]
\includegraphics[scale=0.29]{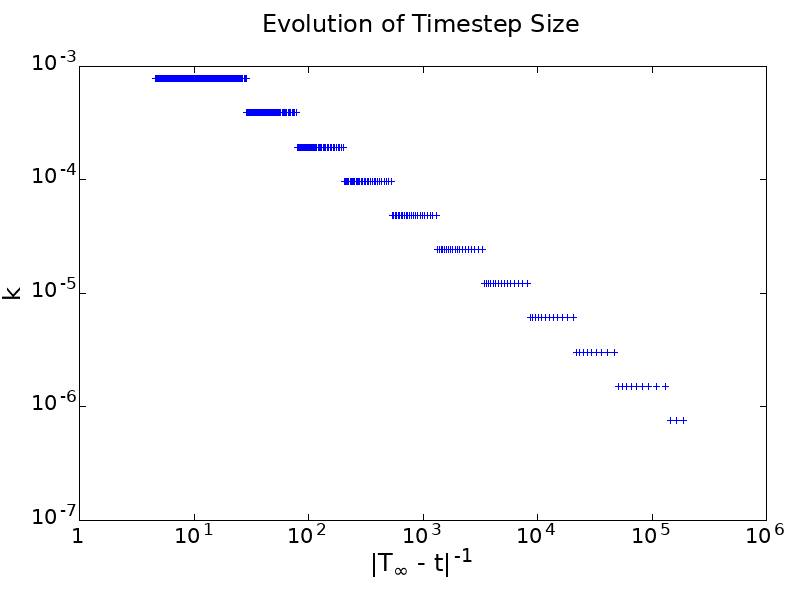}\hfill \includegraphics[scale=0.29]{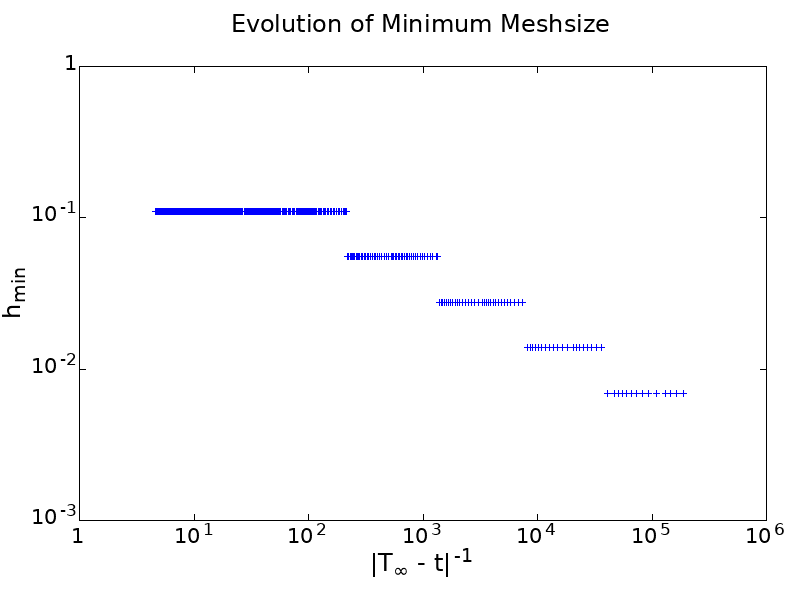}
\caption{Example 1: Behaviour of the numerical solution $\|U(t)\|$ (top left), the numerical blow-up rate $\gamma(t)$ (top right), the time step size $k$ (bottom left), and the minimum mesh size $h_{\min}$ (bottom right) as $t \to T_{\infty}$.}
\label{example1dataplots}
\end{figure}

For the final computational run for $r = 3$ we also include further data which may be of interest in Figure \ref{example1dataplots} with all the quantities of interest plotted against the inverse of the distance from the (approximate) blow-up time $|T_{\infty} - t|^{-1}$ on a logarithmic scale. Firstly, we plot the $\Linf$ norm of the numerical solution $U(t)$ from which it looks like
$(T_{\infty}-t)\|U(t)\| \to C$ as $t \to T_{\infty}$
for some constant $C > 0$ as expected theoretically; this can be seen more readily in the plot of the numerical blow-up rate (see Remark \ref{blow-uprateremark}) which confirms that $\gamma(t) \to 1$ as $t \to T_{\infty}$. These results imply that for the given initial condition the solution is not initially in blow-up phase (since $\gamma(t) \ll 1$ for $t\approx 0$)  and that the nonlinearity needs a significant amount of time before it overcomes the dynamics of the diffusion (expressed by the Laplacian). Out of interest, we also plot the distribution of time step sizes $k$ and minimum mesh sizes $h_{\min} := \min_{K \in \mathcal{T}} h_K$ for this run; other than the fact that both $k$ and $h_{\min}$ register a decrease graded towards $T_{\infty}$ as $t \to T_{\infty}$, it is not obvious to draw a general conclusion. Although not plotted, we did observe that larger temporal polynomial degrees correspond to a larger variability in the distribution of time steps.

\subsection{Example 2: Manifold Blowup}

Let $\Omega = (-10,10)^2$, $\diff = 1$, $f(u) = u^2$ and the ``volcano'' type initial condition be prescribed by the function $u_0(x,y) = 10(x^2+y^2)\exp(-(x^2+y^2)/2)$. The blow-up set for this example is a circle of radius $\sqrt{2}$ centered at the origin which presents two challenges for the adaptive algorithm: firstly, it is a one dimensional manifold rather than a point singularity, so it is going to require more degrees of freedom to resolve; secondly, the blow-up set is not aligned with the mesh and so if resolution of the blow-up set is lost, it could cause the numerical method to fail. Both of these points imply that this example constitutes a good test of the spatial capabilities of the adaptive algorithm. We remark that as the nonlinearity here is the same as in Example 1, $\delta_m$ is again the smallest root of \eqref{deltaquad}.

We again use polynomials of degree $p = 8$ in space with as small a spatial threshold ${\tt stol}^+$ as is permissible. Moreover, we allow the polynomial degree in time ($r = 1$, 2) and the temporal threshold ${\tt ttol}$ to vary in order to see the impact this has on how close we can get to the blow-up time. Applying the procedure discussed in Remark \ref{blow-uprateremark}, we obtain an approximate blow-up time of $T_{\infty} \approx 0.16646111$ for this numerical example. The results, given in Table \ref{example2results}, clearly illustrate the power of increasing the temporal polynomial degree with $\|U(t^-_N)\|$ being of order $10^6$ in 948 time steps for $r = 1$ but in only 221 time steps for $r = 2$ with similar results observed for the blow-up time error $|T_{\infty} - t_N|$. We were unable to observe further progress for $r = 2$  after the second to last computational run because of significant memory requirements;  indeed, the final numerical run had $3\times 10^6$ spatial degrees of freedom and $r+1=3$ temporal degrees of freedom for a total of $9\times 10^6$ total degrees of freedom at the end of the computation. We also include, in Figure \ref{example2solutions}, the spatial meshes and solution contour plots at the initial and final times from the last computational run for $r = 2$ which show that the space estimator is choosing to heavily refine the area in the vicinity of the circular singularity and coarsen elsewhere as expected.

\begin{figure}[h]
\centering
\begin{tabular}{c@{\hspace{1ex}}c}
\includegraphics[scale=0.41]{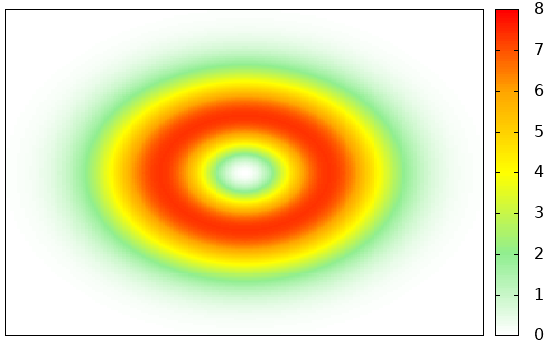}&\includegraphics[scale=0.41]{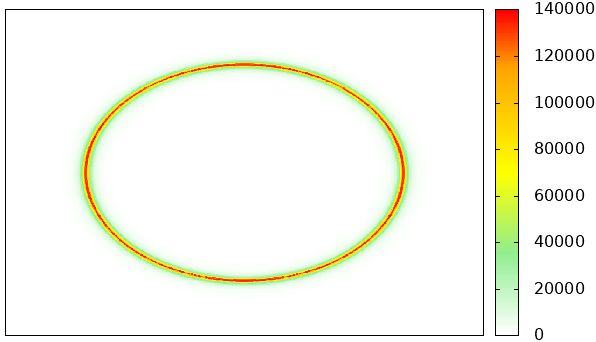}\\[3ex]
\includegraphics[scale=0.39]{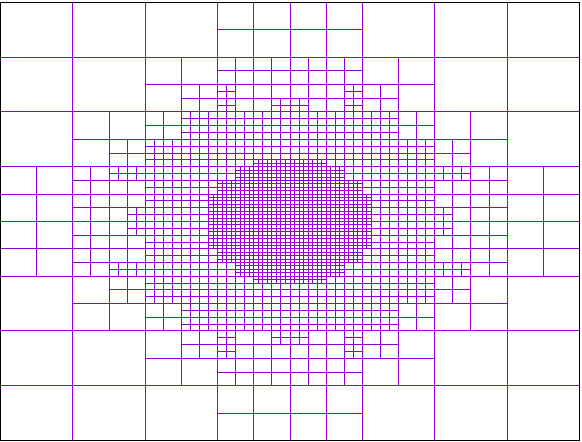}&\includegraphics[scale=0.39]{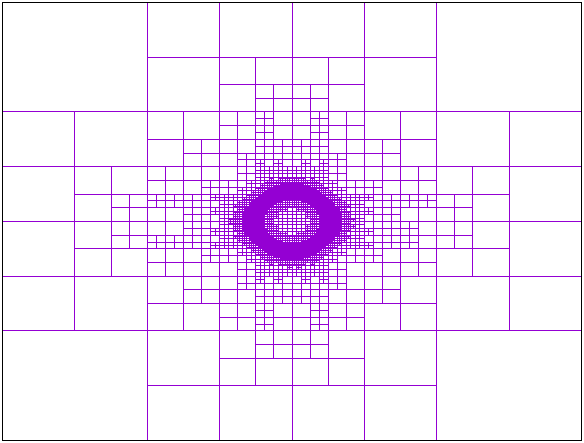}\\[3ex]
\multicolumn{2}{c}{\includegraphics[scale=0.47]{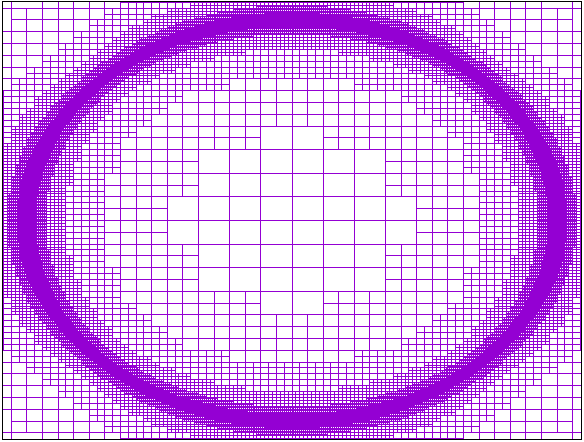}}
\end{tabular}
\caption{Example 2: The initial solution (top left), the final solution (top right), the initial mesh (center left), the final mesh (center right), and final mesh in the vicinity of the singularity (bottom).}
\label{example2solutions}
\end{figure}

\begin{table}
\begin{center}
\begin{tabular}{cccc} \toprule
$N$  & 	$\|U(t^-_N)\|$    & 	$t_N$ & $|T_{\infty}-t_N|$\\ \midrule 
14 & 48.0293 & 0.143750 &  2.27e-2  \\ 
38 & 267.104 & 0.162500 & 3.96e-3 \\  
88 &	1014.01 & 0.165430 & 1.03e-3 \\  
198 &	4147.83 & 0.166211 & 2.50e-4 \\  
435 & 24182.8 & 0.166418 & 4.27e-5 \\  
948 &	97237.1 & 0.166451 & 1.06e-5 \\ \bottomrule 
\end{tabular}\hfill
\begin{tabular}{cccc} \toprule
$N$  & 	$\|U(t^-_N)\|$    & 	$t_N$ & $|T_{\infty}-t_N|$\\ \midrule 
3 & 17.1009	& 0.100000 & 6.65e-2 \\ 
16 & 192.648 & 0.160938 & 5.52e-3 \\ 
28 & 439.572 & 0.164063 & 2.40e-3 \\ 
64 & 6794.02  & 0.166309 & 1.53e-4  \\ 
221 & 113665 & 0.166452 & 9.09e-6 \\ 
{\it 404} & {\it136719} & {\it0.166454} & {\it 7.56e-6} \\ \bottomrule
\end{tabular}\\[3ex]
\end{center}
\caption{Example 2: Results of the numerical experiments for $r = 1$ (left) and $r = 2$ (right). The \textit{italic} result indicates a lack of progression due to memory.}\label{example2results}
\end{table}

For the final computational run for $r = 2$ we also include further data which may be of interest in Figure \ref{example2dataplots} with all the quantities of interest plotted against the inverse of the distance from the (approximate) blow-up time $|T_{\infty} - t|^{-1}$ on a logarithmic scale. Firstly, we plot the numerical blow-up rate (see Remark \ref{blow-uprateremark}) from which we observe that $\gamma(t) \to 1$ as $t \to T_{\infty}$; unlike in Example 1, we notice some oscillations likely caused by the heavy amount of mesh refinement taking place, and the numerical blow-up rate also briefly rises above~1 near the end of the computation due to inaccuracies in our approximation of $T_{\infty}$. Furthermore, in contrast to Example 1, the solution to \eqref{model_strong} for Example 2 seems to be in blow-up phase right from the initial stages as $\gamma(t) \approx 0.9$ for $t \approx 0$. Out of interest, we also plot the distribution of spatial degrees of freedom (DoFS), time step sizes $k$ and minimum mesh sizes $h_{\min} := \min_{K \in \mathcal{T}} h_K$ for this run. In contrast to Example 1, we observe a significant increase in spatial degrees of freedom as $t \to T_\infty$ from around 150,000 required to resolve the initial condition to over $3\times 10^6$ at termination though this is to be expected given that the singularity in Example 2 is one dimension larger than the singularity in Example 1; this can also be seen by comparing the plots of the minimum mesh sizes with those from Example 1, cf. Figure \ref{example1dataplots}. Broadly speaking, the distribution of time steps is comparable to that observed in Example 1, however, Example 1 has more more variability due to using $r = 3$.

\begin{figure}
\centering
\includegraphics[scale=0.295]{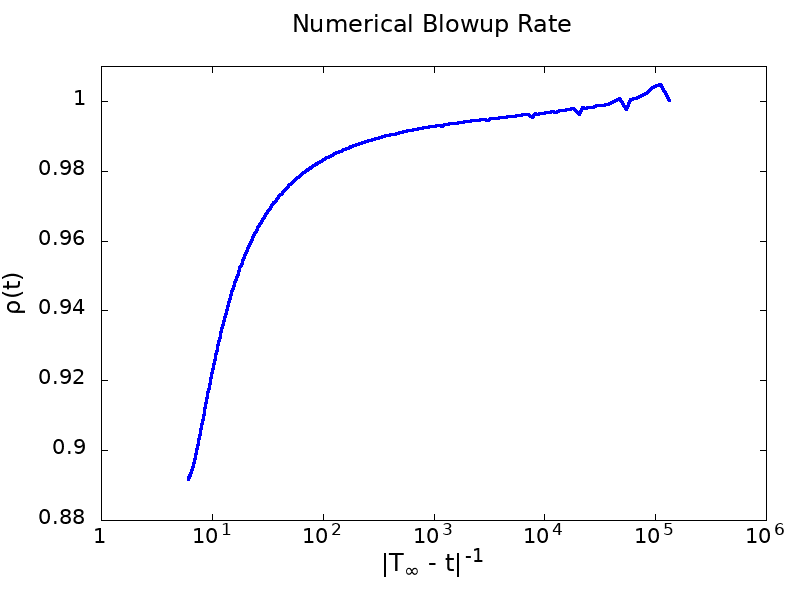} \includegraphics[scale=0.295]{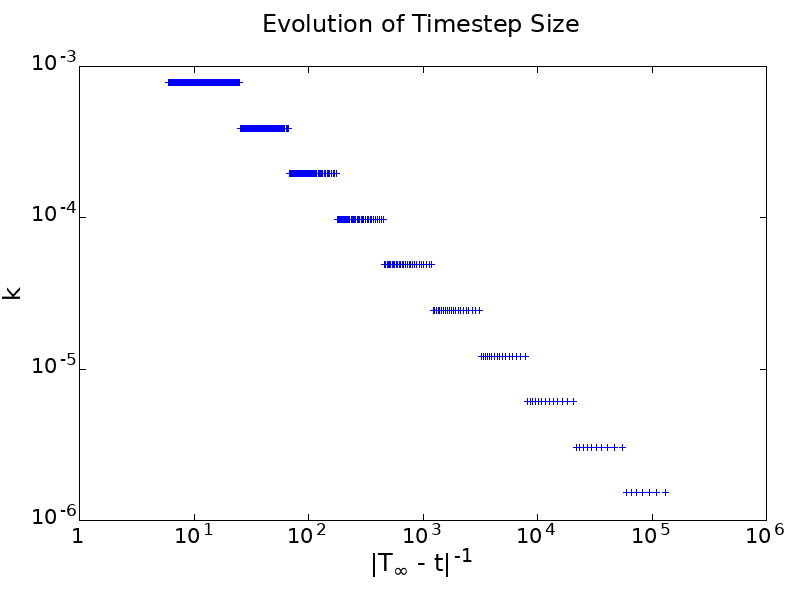}
\\[3ex]
\includegraphics[scale=0.295]{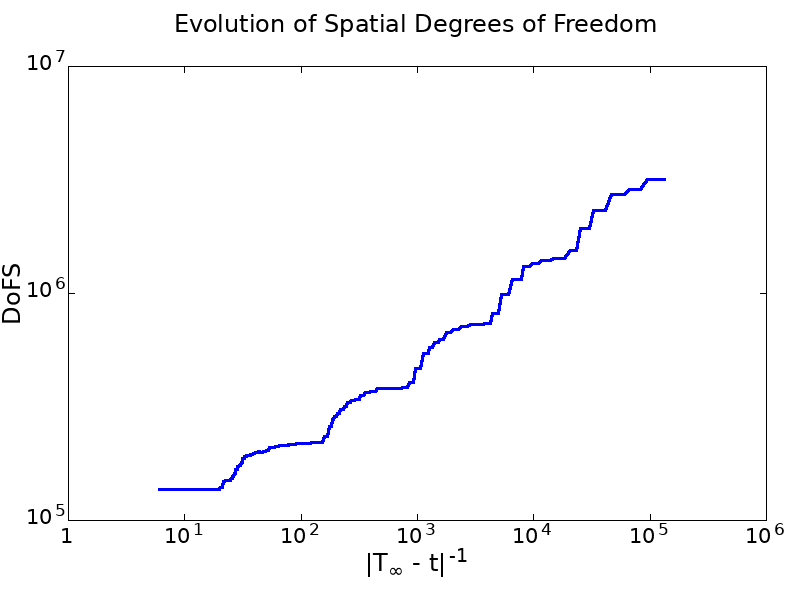} \includegraphics[scale=0.295]{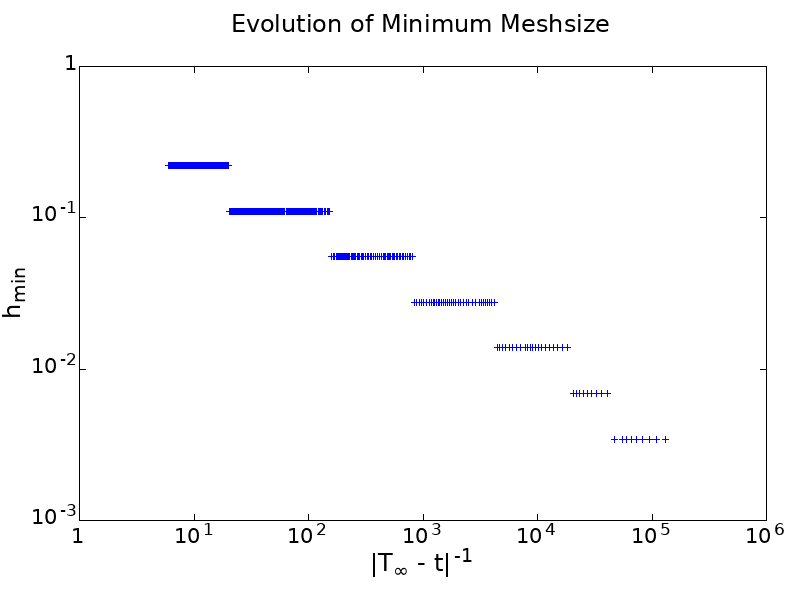}
\caption{Example 2: The numerical blow-up rate $\gamma(t)$ (top left), the time step size $k$ (top right), spatial degrees of freedom (bottom left) and the minimum mesh size $h_{\min}$ (bottom right) as $t \to T_{\infty}$.}
\label{example2dataplots}
\end{figure}

\subsection{Example 3:  Point Blowup ($hp$-adaptivity)}

For this example, we use the same data as Example 1 (point blow-up) with the intention of trying out an $hp$-type adaptive strategy utilising the \emph{a posteriori} error estimator. The adaptive algorithm is also similar to the previous two examples, however, now we employ a \emph{(linear) temporal polynomial degree distribution} according to
$$r_m = \max\{0,\,\lceil r_0 + \sigma \log(\nicefrac{k_m}{k_0})\rceil \},\qquad m\ge 1,$$
where $\sigma = 0.47$ is a parameter which determines at what time step size the polynomial degree is decreased and $\lceil \cdot \rceil$ is the rounding up function. This idea is inspired by an \emph{a priori} $hp$-strategy for the resolution of (algebraic) start-up singularities in linear parabolic problems which has been proven to yield exponential convergence rates; see, e.g.~\cite{SS00}. As a decrease in the temporal polynomial degree close to the blow-up time could end up with the temporal error actually increasing on that time interval, we set $k_m \gets \nicefrac{k_{m-1}}{4}$ if the polynomial degree changes in time instead of the usual $k_m \gets \nicefrac{k_{m-1}}{2}$. We remark that this idea is fundamentally different from \cite{KMW16} where the temporal polynomial degree is increased towards the blow-up time (which has already been proven to yield exponential convergence to the blow-up time in ODEs \cite{WY20}) whereas here we decrease the polynomial degree towards the blow-up time thereby saving a considerable number of degrees of freedom as $k_m\to 0$.

As in the previous examples, we choose polynomials of degree $p=8$ in space with a spatial threshold $\tt{stol}^+$ which is chosen as small as is permissible. In time, our experiments begin with polynomials of degree $r_0 = 1$ and a temporal threshold of ${\tt ttol} = 10^{-3}$; we then do three additional computational runs by reducing the temporal threshold by a factor of 100 while increasing the initial temporal polynomial degree by 1 until $r_0 = 4$ and ${\tt ttol} = 10^{-9}$. The results of these runs, given in Figure \ref{example3plots}, show that we initially observe \emph{exponential convergence} to the blow-up time as expected, however, the final data point breaks this trend which indicates spatial effects become dominant in the final computational run; here, we also plot the distribution of time step sizes and polynomial degrees versus the inverse of the distance to the blow-up time $|T_{\infty} - t|^{-1}$ showing that the time estimator and our proposed polynomial degree strategy yields a reasonable distribution of time step sizes and polynomial degrees graded towards the blow-up time.

\begin{figure}
\centering
\includegraphics[scale=0.3]{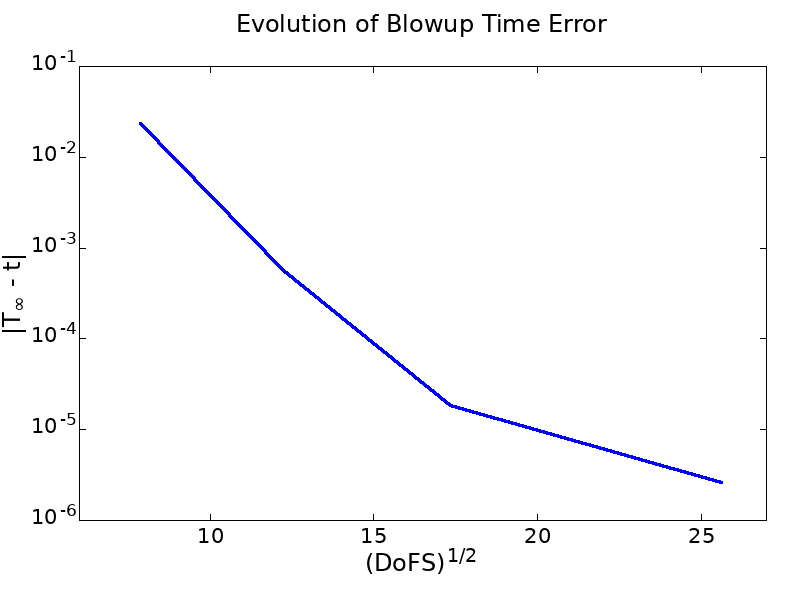}\includegraphics[scale=0.3]{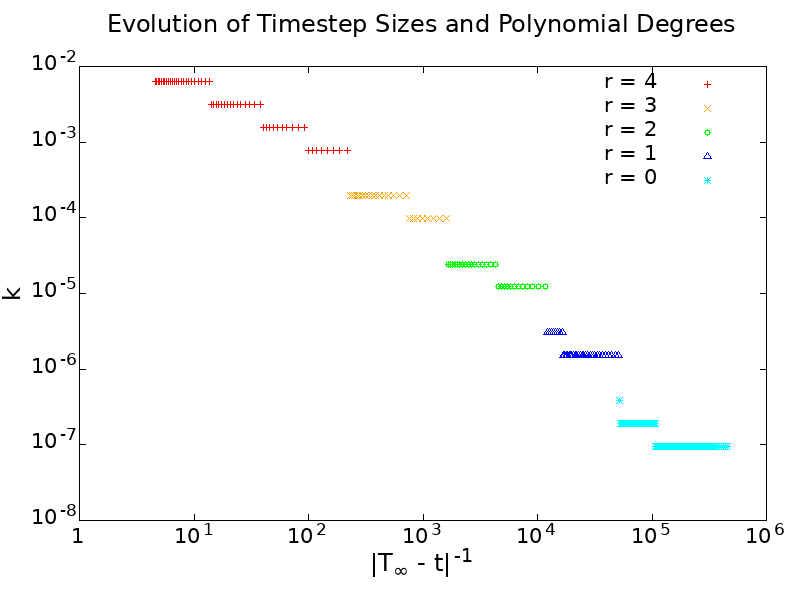}
\caption{Example 3: Blow-up time error $|T_{\infty} - t|$ vs. $\sqrt{\textrm{(Temporal) DoFS}}$ (left) and the distribution of the time step sizes $k$ and polynomial degrees $r$ as $t \to T_{\infty}$ for the final computational run (right).}
\label{example3plots}
\end{figure}

\section{Conclusions}\label{sc:concl}

In this work we have derived a \emph{conditional a posteriori} error bound (Theorem \ref{maintheorem}) for the dG-cG discretisation \eqref{dg_method} of the semilinear heat equation \eqref{model_strong}. Numerical experiments indicate that the a posteriori error estimator performs well when driving adaptivity for problems with two different types of singularities: single point (Example 1) and manifold (Example 2) with temporal refinement graded towards the blow-up time and spatial refinement graded towards the singularity in both cases. In addition, we observed exponential convergence towards the blow-up time when using a temporal $hp$-version dG time stepping scheme for the blow-up problem with the point singularity (Example 3); here, the exploitation of a highly effective distribution of temporal degrees of freedom permits us to start with a relatively high polynomial degree in the initial phase (yielding spectrally accurate approximations) and then reduce the approximation order appropriately closer to the blow-up time. The primary practical limitation of the method appears to be floating point errors in the computation of the space derivative estimator $\displaystyle{\accentset{\bigcdot}{\eta}}_{\mathrm{space}}$. Future work includes the implementation of space-time $hp$ adaptivity for these types of problems.

\bibliographystyle{amsplain}
\bibliography{paper1}
\end{document}